\documentclass[12pt,oneside]{amsart}
\usepackage{calc,ifthen,geometry}
  \usepackage{verbatim} 
\usepackage{graphicx}
\usepackage[usenames,dvipsnames,svgnames,table]{xcolor}
      \usepackage{amssymb}
\usepackage{ulem}       
\usepackage{amsmath, amsthm, amssymb}
\usepackage{xspace, enumerate}
\usepackage{mathrsfs}
\usepackage[pdftex,bookmarks,colorlinks,breaklinks]{hyperref}  
\hypersetup{linkcolor=blue,citecolor=red,filecolor=dullmagenta,urlcolor=darkblue}
 \usepackage{color}	
 
 \usepackage{mathabx}
\newtheorem{theorem}{Theorem}[section]
\newtheorem{corollary}[theorem]{Corollary}
\newtheorem{prop}[theorem]{Proposition}
\newtheorem{lemma}[theorem]{Lemma}

\newtheorem{question}{Problem}

\newcommand{\cG}{\mathcal{G}}
\newcommand{\cF}{\mathcal{F}}

  \newcommand\mult{\operatorname{mult}}

\newcommand{\cc}{\mathcal}

      \makeatletter
      \def\@setcopyright{}
      \def\serieslogo@{}
      \makeatother

\begin{document}
   \author{Amin Bahmanian}
   \address{Department of Mathematics
 Illinois State University, Normal, IL USA   61790-4520}

   \title[Extending  edge-colorings of complete hypergraphs] 
   {Extending  edge-colorings of complete hypergraphs into regular colorings} 

   \begin{abstract} Let $\binom{X}{h}$ be the collection of all $h$-subsets of an $n$-set $X\supseteq Y$. Given a coloring (partition) of a set $S\subseteq \binom{X}{h}$, we  are interested in finding  conditions under which this coloring  is extendible to a coloring of $\binom{X}{h}$ so that  the number of times each element of $X$ appears in each color class (all sets of the same color)  is the same number $r$. The case $S=\varnothing, r=1$ was studied by Sylvester in the 18th century, and remained open until the 1970s. The case $h=2,r=1$ is extensively studied in the literature and is closely related to completing partial symmetric Latin squares. 
   
For $S=\binom{Y}{h}$, we settle the cases $h=4, |X|\geq 4.847323|Y|$, and $h=5, |X|\geq 6.285214|Y|$ completely.  Moreover, we make partial progress toward solving the  case  where $S=\binom{X}{h}\backslash \binom{Y}{h}$. 
 These results can be seen as extensions of the famous Baranyai's theorem, and make  progress toward settling a 40-year-old problem posed by Cameron.

   \end{abstract}


   \maketitle   
\section{Introduction}
Suppose that we have been entrusted to color (or partition) the collection $\binom{[n]}{h}$ of all $h$-subsets of the $n$-set $[n]:=\{1,\dots,n\}$ so that 
 the number of times each element of $[n]$ appears in each color class (all sets of the same color)  is exactly $r$. Such a coloring is called an {\it $r$-factorization} of $\binom{[n]}{h}$. A solution for the case $n=6,h=3,r=1$ with 10 colors  is given below.
\\

{\color{Green}$\{1, 4, 5\}, \{2, 3, 6\}$} \quad 
{\color{Apricot}$\{1, 2, 4\}, \{3, 5, 6\}$} \quad
{\color{DarkOrchid} $\{1, 3, 6\}, \{2, 4, 5\}$} \quad
{\color{Maroon}  $\{1, 2, 3\}, \{4, 5, 6\}$} \quad
{\color{RoyalBlue}  $\{1, 2, 5\}, \{3, 4, 6\}$}

 {\color{Turquoise}  $\{1, 5, 6\}, \{2, 3, 4\}$} \quad
{\color{black}  $\{1, 3, 5\}, \{2, 4, 6\}$} \quad
{\color{RawSienna}  $\{1, 4, 6\}, \{2, 3, 5\}$} \quad
{\color{ProcessBlue} $\{1, 3, 4\}, \{2, 5, 6\}$} \quad
{\color{red}  $\{1, 2, 6\}, \{3, 4, 5\}$} 
\\

Note that the number of times each element of $[n]$ appears in $\binom{[n]}{h}$ is $\binom{n-1}{h-1}$. Thus, for $\binom{[n]}{h}$ to be $r$-factorable, it is clear that (i) $r$ must divide $\binom{n-1}{h-1}$. In addition, a simple double counting argument shows that (ii) $h$ must divide $rn$.
One may wonder if conditions (i) and (ii) are also sufficient for $\binom{[n]}{h}$ to be $r$-factorable. In  the 18th century, Sylvester considered the case  $r=1$ of this problem which remained open until the 1970s when Baranyai solved this 120-year-old problem completely \cite{MR0416986}. In fact, Baranyai proved a far more general result which, in particular, implies that $\binom{[n]}{h}$ is $r$-factorable if and only if $h| rn$ and $r | \binom{n-1}{h-1}$. 

We are interested in a Sudoku-type version of Baranyai's theorem. A {\it partial $r$-factorization} of a set $S\subseteq\binom{[n]}{h}$ is a coloring of $S$ with at most $\binom{n-1}{h-1}/r$ colors so  that  the number of times each element of $[n]$ appears in each color class  is at most $r$. Note that a color class may be empty.
\begin{question} Under what conditions can a partial $r$-factorization of  $S\subseteq \binom{[n]}{h}$  be extended to an $r$-factorization of  $\binom{[n]}{h}$?
\end{question}
We are given a coloring of  a subset $S\subseteq \binom{[n]}{h}$, and our task is to  complete the coloring. In other words, we need to color $T:=\binom{[n]}{h}\backslash S$ so that the coloring of $S\cup T$ provides an $r$-factorization of $\binom{[n]}{h}$. Baranyai's theorem settles the case when $S=\varnothing$. 
A partial 4-factorization of $\binom{[9]}{3}$ is given below (Here we  abbreviate a set $\{a,b,c\}$ to $abc$).
\\

\begin{center}
{\color{red} 156, 248, 379, 126, 348, 579, 127, 349, 568, 124, 389, 567}

{\color{RoyalBlue}148, 267, 359, 168, 279, 345, 159, 278, 346, 134, 259}
 
{\color{Green}128, 347, 569,  178, 249, 356, 169, 247, 358, 123}

{\color{DarkOrchid}146, 239, 578, 137, 289, 456, 136, 257}

{\color{black}129, 367, 458, 125, 368, 479, 147, 258, 369, 157}

{\color{RawSienna}189, 246, 357, 158, 237, 469, 138, 245, 679, 139, 268}

{\color{Turquoise}145, 236, 789, 167, 238, 459, 149, 256, 378, 135, 269, 478}
\end{center}

It is not too difficult to extend this to the following 4-factorization.
\\

\begin{center}
{\color{red} 156, 248, 379, 126, 348, 579, 127, 349, 568, 124, 389, 567}

{\color{RoyalBlue}148, 267, 359, 168, 279, 345, 159, 278, 346, 134, 259, 678}
 
{\color{Green}128, 347, 569,  178, 249, 356, 169, 247, 358, 123, 467, 589}

{\color{DarkOrchid}146, 239, 578, 137, 289, 456, 136, 257, 489, 179, 235, 468}

{\color{black}129, 367, 458, 125, 368, 479, 147, 258, 369, 157, 234, 689}

{\color{RawSienna}189, 246, 357, 158, 237, 469, 138, 245, 679, 139, 268, 457}

{\color{Turquoise}145, 236, 789, 167, 238, 459, 149, 256, 378, 135, 269, 478}

\end{center}

\hspace{1in}

The case $h=2,r=1$ of Problem 1 is closely related to completing  partial Latin squares, (see  Lindner's excellent survey \cite{MR1096296}).
A special case of Problem 1 when $r=1$, and the partial factorization is a 1-factorization of  $\binom{[m]}{h}$ for some $m<n$, was studied by Cruse (for $h=2$) \cite{MR0329925}, Cameron \cite{MR0419245}, and Baranyai and Brouwer \cite{BaranBrouwer77}. Baranyai and Brouwer conjectured that a 1-factorization of $\binom{[m]}{h}$ can be extended to a 1-factorization of $\binom{[n]}{h}$ if and only if $n\geq 2m$ and $h$ divides $m,n$.  H{\"a}ggkvist and Hellgren \cite{MR1249714} gave a beautiful proof of this conjecture. For further generalizations of H{\"a}ggkvist-Hellgren's result, we refer the reader to two recent papers by the author and Newman \cite{CCABahNew, MR3512664} in which
 extending $r$-factorizations of $\binom{[m]}{h}$ to  $s$-factorizations of $\binom{[n]}{h}$ is studied (for $s\geq r$).

At this point, it should be clear to the reader that the  1-factorization of $\binom{[6]}{3}$ in the first example,  can not be extended to a 1-factorization of $\binom{[9]}{3}$, but it can be extended to a 1-factorization of $\binom{[12]}{3}$.

Like most results in the literature, our primary focus is the case where $S=\binom{[m]}{h}$ (for some $m<n$). However, unlike those, here we do not require the given partial factorization to be a factorization itself. 
In this case, Problem 1 was settled by Rodger and Wantland over 20 years ago for $h=2$ \cite{MR1315436}, and recently by the author and Rodger  for $h=3, n\geq 3.414214m$ \cite{MR3056885}. In this paper, we settle the cases $h=4, n\geq 4.847323m$ and $h=5, n\geq6.285214m$. The major obstacle from $h=2$ to $h\geq 3$ stems from the natural difficulty of generalizing a graph theoretic result  to hypergraphs. 

Note that, in order to extend a partial $r$-factorization of $\binom{[m]}{h}$ to an $r$-factorization of  $\binom{[n]}{h}$ (for $n\geq m$), it is clearly necessary that $r|\binom{n-1}{h-1}$, $h | rn$. Let $\chi (m,h,r)$ be the smallest $n$ such that {\it any} partial $r$-factorization of $\binom{[m]}{h}$ satisfying $r|\binom{n-1}{h-1}$, $h | rn$ can be extended to an $r$-factorization of  $\binom{[n]}{h}$. Combining the results of this paper with those of \cite{CCABahNew, MR3512664, MR3056885}, it can be easily   shown that  $2m\leq \chi (m,3,r)\leq 3.414214m, 2m\leq \chi (m,4,r)\leq 4.847323m$, and $2m\leq \chi (m,5,r)\leq 6.285214m$. 

Last but not least, we shall consider Problem 1 in the case when $S=\binom{[n]}{h}\backslash \binom{[m]}{h}$. In this direction, we solve a variation of the problem when we allow sets of size less than $h$, and in our extension of the coloring we also extend the sets of size less than $h$ to sets of size $h$. 

The paper is self-contained and all the preliminaries are given in Section \ref{prelsec}. In section \ref{piecesech}, we shall consider Problem 1 in the case when $S=\binom{[n]}{h}\backslash \binom{[m]}{h}$. The cases $h=4,5$ are discussed in detail in Sections \ref{h4emb}, \ref{h5emb}, respectively. We conclude the paper with some open problems. 
\section{Notation and Tools} \label{prelsec}
A {\it hypergraph} $\mathcal G$ is a pair $(V(\mathcal G),E(\mathcal G))$ where $V(\mathcal G)$ is a finite set called the {\it vertex} set, $E(\mathcal G)$ is the {\it edge} multiset, where every edge is itself a multi-subset of $V(\mathcal G)$. This means that not only can an edge  occur multiple times in $E(\mathcal G)$, but also each vertex can have multiple occurrences within an edge.  By an edge of the form $\{u_1^{m_1},u_2^{m_2},\dots,u_s^{m_s}\}$, we mean an edge in which vertex $u_i$ occurs $m_i$ times for $1\leq i\leq r$. 
The total number of occurrences of a  vertex $v$ among all edges of $E(\mathcal G)$ is called the {\it degree}, $\deg_{\mathcal G}(v)$ of $v$ in $\mathcal G$. The {\it multiplicity} of an edge $e$ in $\mathcal G$, written $\mult_{\mathcal G} (e)$, is the number of repetitions of $e$ in $E(\mathcal G)$ (note that $E(\mathcal G)$ is a multiset, so an edge may appear multiple times). If $\{u_1^{m_1},u_2^{m_2},\dots,u_s^{m_s}\}$ is an edge in $\mathcal G$, then we abbreviate $\mult_{\mathcal G}(\{u_1^{m_1},u_2^{m_2},\dots,u_s^{m_s}\})$ to $\mult_{\mathcal G}(u_1^{m_1},u_2^{m_2},\dots,u_s^{m_s})$. 
If $U_1,\dots,U_s$ are multi-subsets of $V(\mathcal G)$, then $\mult_{\mathcal G}(U_1,\dots,U_s)$ means $\mult_{\mathcal G}(\bigcup_{i=1}^s U_i)$, where the union of $U_i$s is the usual union of multisets. Whenever it is not ambiguous, we drop the subscripts; for example we write $\deg(v)$ and $\mult(e)$ instead of $\deg_{\mathcal G}(v)$ and $\mult_{\mathcal G} (e)$, respectively. 

For $h\in \mathbb{N}$, $\mathcal G$ is said to be $h$-{\it uniform} if $|e|=h$ for each $e\in E$, and an $h$-factor in a hypergraph $\mathcal G$ is a spanning  $h$-regular sub-hypergraph. An {\it $h$-factorization}  is a partition of the edge set of $\mathcal G$ into $h$-factors. 
The hypergraph $K_n^h:=(V,\binom{V}{h})$ with $|V|=n$  is called a  {\it complete} $h$-uniform hypergraph. 
A {\it $k$-edge-coloring} of $\mathcal G$ is a mapping $f:V(\mathcal G)\rightarrow [k]$  and color class $i$ of $\mathcal G$, written $\mathcal G(i)$, is the sub-hypergraph of $\mathcal G$ induced by the edges of color $i$.

Let $\mathcal G$ be a hypergraph, let $U$ be
some finite set, and let $\Psi : V(\mathcal G) \to U$ be a surjective mapping.  The map $\Psi$ extends
naturally to $E(\mathcal G)$.  For $A \in E(\mathcal G)$ we define $\Psi(A) = \{\Psi(x) : x \in A\}$.
Note that $\Psi$ need not be injective, and $A$ may be a multiset.  Then we define the
hypergraph $\mathcal F$ by taking $V(\mathcal F)=U$ and $E(\mathcal F)=\{ \Psi(A) : A \in E(\mathcal G) \}$.  We say
that $\mathcal F$ is an {\it amalgamation} of $\mathcal G$, and that $\mathcal G$ is a {\it detachment} of
$\mathcal F$.  Associated with $\Psi$ is a (number) function $g$ defined by $g(u)=|\Psi^{-1}(u)|$; to be
more specific we will say that $\mathcal G$ is a $g$-detachment of $\mathcal F$.  Then $\mathcal G$ has $\sum_{u\in
  V(\mathcal F)} g(u)$ vertices.  Note that $\Psi$ induces a bijection between the edges of $\mathcal F$ and
the edges of $\mathcal G$, and that this bijection preserves the size of an edge.  We adopt the
convention that it preserves the color also, so that if we amalgamate or detach an
edge-colored hypergraph the amalgamation or detachment preserves the same coloring on the
edges.  We make explicit a straightforward observation: Given $\mathcal G$, $V(\mathcal F)$ and $\Psi$ the
amalgamation is uniquely determined, but given $\mathcal F$, $V(\mathcal G)$ and $\Psi$ the detachment is in
general far from uniquely determined.

There are quite a lot of other papers on  amalgamations  and some highlights include  \cite{MR1863393, MR592087, MR746544, MR920647,  MR820838, MR2325799, MR916377, MR1315436}. 

Given an edge-colored hypergraph $\cc F$, we are interested in finding a detachment $\cc G$ obtained by splitting each vertex of $\cc F$ into a prescribed number of vertices in $\cc G$ so that (i) the degree of each vertex in each color class of $\cc F$ is shared evenly among the subvertices in the same color class in $\cc G$, and (ii) the multiplicity of each edge in $\cc F$ is shared evenly among the subvertices in $\cc G$. 
The following theorem, which is a  special case of a general result in \cite{MR2942724}, guarantees the existence of such detachment (Here $x \approx y$  means $\lfloor y \rfloor \leq x \leq  \lceil y \rceil$). 
\begin{theorem}\textup{(Bahmanian \cite[Theorem 4.1]{MR2942724})}\label{mainthhypgen1}
Let $\mathcal F$ be a $k$-edge-colored hypergraph and let $g:V(\mathcal F)\rightarrow {\mathbb N}$. Then there exists a  $g$-detachment $\mathcal G$ (possibly with multiple edges) of $\mathcal F$ whose edges are all sets, with amalgamation function $\Psi:V(\mathcal G)\rightarrow V(\mathcal F)$,  $g$  being the number function associated with $\Psi$, such that
\begin{enumerate}
\item [\textup{(F1)}] for each $u\in V(\mathcal F)$, each $v\in \Psi^{-1}(u)$ and $i\in [k]$,
$$\deg_{\cc G(i)}(v) \approx \frac{\deg_{\cc F(i)}(u)}{g(u)};$$ 
\item [\textup{(F2)}] for  distinct $u_1,\dots,u_s\in V(\mathcal F)$ and  $U_i\subseteq \Psi^{-1}(u_i)$ with $|U_i|=m_i\leq g(u_i)$ for $i\in [s]$, 
$$\mult_{\cc G}(U_1,\dots,U_s) \approx \frac{\mult_{\cc F} (u_1^{m_1},\dots,u_s^{m_s})}{\Pi_{i=1}^s\binom {g(u_i)}{m_i}}.$$ 
\end{enumerate}
\end{theorem}
Let $\widetilde {K_m^h}$ be the hypergraph obtained by adding a new vertex $u$ and new edges to $K_m^h$ so that 
$$\mult(u^{i},W)=\binom{n-m}{i} \mbox{ for each }i\in [h], \mbox{and } W\subseteq V(K_m^h) \mbox{ with } |W|=h-i.$$
In other words, $\widetilde {K_m^h}$ is an amalgamation of $K_n^h$, obtained by identifying an arbitrary set of $n-m$ vertices in $K_n^h$.  

An immediate consequence of Theorem \ref{mainthhypgen1} is the following.
\begin{corollary} \label{maindetachcor}
Let $k:=\binom{n-1}{h-1}/r\in \mathbb{N}$. A partial $r$-factorization of $K_m^h$ can be extended to an $r$-factorization of $K_n^h$ if and only if the new edges of $\cc F:=\widetilde {K_m^h}$
can be colored so that  
\begin{equation}  \label{maindetachcoreqrr}
\forall i\in [k] \quad \quad  \deg_{\cc F(i)}(v) = \left \{ \begin{array}{ll}
r & \mbox { if } v\neq u,  \\
r(n-m) & \mbox { if } v=u. \end{array} \right.\\
\end{equation}
\end{corollary}
\begin{proof} First, suppose that a partial $r$-factorization of $K_m^h$ can be extended to an $r$-factorization of $K_n^h$.  By amalgamating  the new  $n-m$ vertices of $K_n^h$ into a single vertex $u$,   we clearly obtain $\cc F$. The $k$-edge-coloring of $K_n^h$ (in which each color class is an $r$-factor) induces a $k$-edge-coloring in $\cc F$ that satisfies  (\ref{maindetachcoreqrr}).

Conversely, suppose that  the edges of  $\cc F$ are colored  so that (\ref{maindetachcoreqrr}) is satisfied.  
Let $g:V(\cc F)\rightarrow {\mathbb N}$ with $g(u)=n-m$, and $g(v)=1$ for  $v\neq u$. By Theorem \ref{mainthhypgen1}, there exists a $g$-detachment $\cc G$ of $\cc F$ such that
\begin{itemize}
\item [(a)] for each $v\in \Psi^{-1}(u)$, and $i\in[k]$
\begin{equation*} 
\deg_{\cc G(i)}(v) \approx \deg_{\cc F(i)}(u)/g(u)=r(n-m)/(n-m)=r.
\end{equation*}
\item [(b)]for   $U\subseteq \Psi^{-1}(u),W\subseteq V(K_m^h)$ with $|U|=i, |W|=h-i$, for $i\in [h]$. 
$$\mult_{\cc G}(U,W) \approx \frac{\mult_{\cc F}(u^{i},W)} { {\binom {g(u)}{i}}}=\frac{\binom{n-m}{i}}{\binom{n-m}{i}}=1$$
\end{itemize}  
By (a), each color class is an $r$-factor, and by (b), $\cc G\cong K_n^h$.
\end{proof}
The following observation will be quite useful throughout the paper. 
\begin{prop} For every $n,m,h\in \mathbb{N}$ with $n\geq m\geq h$, 
\begin{equation}\label{easydc1}
\binom{n}{h}=\sum_{i=0}^h\binom{m}{i}\binom{n-m}{h-i}.
\end{equation}
\begin{equation}\label{easydc2}
m[\binom{n-1}{h-1}-\binom{m-1}{h-1}]=\sum_{i=1}^{h-1} i\binom{m}{i}\binom{n-m}{h-i}.
\end{equation}
\end{prop}
\begin{proof}
The proof of (\ref{easydc1}) is straightforward. Let $\mathcal F$ be a hypergraph with vertex set $\{u,v\}$ such that  $\mult(u^{i},v^{h-i})=\binom{m}{i}\binom{n-m}{h-i}$ for $0\leq i\leq h-1$. Note that $\cF$ is an amalgamation of the hypergraph $\cG$ with edge set $\binom{X}{h}\backslash \binom{U}{h}$ where $|X|=n, |U|=m$. Double counting the degree of $u$  proves (\ref{easydc2}):
\begin{eqnarray*}
\sum_{i=1}^{h-1} i\binom{m}{i}\binom{n-m}{h-i}=\deg_\cF(u)=\sum_{u\in U}d_\cG(u)=m[\binom{n-1}{h-1}-\binom{m-1}{h-1}].
\end{eqnarray*}
\end{proof}

In order to avoid trivial cases, throughout the rest of this paper we assume that $m>h$. 
\section{Arbitrary $h$} \label{piecesech}
If we replace every edge $e$ of a hypergraph $\cc G$ by  $\lambda$ copies of $e$, then we denote the new hypergraph by $\lambda \cc G$. For hypergraphs $\cc G_1,\dots, \cc G_t$ with the same vertex set $V$, we define their {\it union}, written $\bigcup_{i=1}^{t}\cc G_i$, to be  the hypergraph with vertex set $V$ and edge set $\bigcup_{i=1}^{t}E(\cc G_i)$. For a hypergraph $\cc G$ and $V\subseteq V(\cc G)$, let $\cc G-V$ be the hypergraph whose vertex set is $V(\cc G)\backslash V$ and whose edge set is $\{e\backslash V|e\in E(\cc G)\}$.

Let $V$ be an arbitrary subset of vertices in $K_n^h$ with  $|V|=m\leq n$. Then $K
_n^h-V\cong \bigcup_{i=0}^{h-1}\binom{m}{i} K_{n-m}^{h-i}$. A {\it partial $r$-factorization} of $\cc H:=K_n^h-V$ is a coloring of the edges of $K_n^h-V$ with at most $\binom{n-1}{h-1}/r$ colors so that for each color $i$,  $\deg_{\cc H(i)}(v)\leq r$ for each vertex of $\cc H$ (Note that $\cc H$ has singleton edges). In the next result, we completely settle the problem of extending a partial $r$-factorization of $K_n^h-V$ to an $r$-factorization of $K_n^h$. Note that here we are not only extending the coloring, but also the edges of size less than $h$ to edges of size $h$. The case $h=3$ was solved in \cite{MR3056885}.

\begin{theorem} \label{piecesthm}
For $V\subseteq V(K_n^h)$ with $|V|=m$, any partial $r$-factorization of $\cc H:=K_n^h-V$
 can be extended to an $r$-factorization of $K_n^h$ if and only if $h | rn$, $r| \binom{n-1}{h-1}$, and for all $i=1,2,\dots, \binom{n-1}{h-1}/r$,
 \begin{equation}\label{neccpiece0}
  d_{\cc H(i)}(v)=r \quad \forall v\in V(\cc H), 
  \end{equation}
\begin{equation}\label{neccpiece}
|E(\cc H(i))|\leq \frac{rn}{h}.
\end{equation}
\end{theorem}
\begin{proof}
To prove the necessity, suppose that a given partial $r$-factorization of $\cc H$ is extended to an $r$-factorization of $K_n^h$. For $K_n^h$ to be $r$-factorable, the  two divisibility conditions are clearly necessary. By extending an  edge $e$ of size $i$ ($i<h$) in $\cc H$ to an edge of size $h$ in $K
_n^h$, the color of $e$ does not change, and so (\ref{neccpiece0}) is necessary.
Since the number of edges in each color class of $K_n^h$ is exactly $rn/h$, the necessity of (\ref{neccpiece}) is implied. 

To prove the sufficiency, suppose that a  partial $r$-factorization of $\cc H$ is given, $h | rn$, $r| \binom{n-1}{h-1}$, and that (\ref{neccpiece0}), (\ref{neccpiece}) are satisfied. Let $k=\binom{n-1}{h-1}$, and let $\cc F=\widetilde {K_{n-m}^h}$. For $0\leq i\leq h$, an edge of {\it type} $u^i$ in $\cc F$ is an edge in $\cc F$  containing $u^i$ but not containing $u^{i+1}$. Note that there are $\binom{m}{i}\binom{n-m}{h-i}$ edges of type $u^i$ in $\cc F$. 

There is a clear one-to-one correspondence between the edges of size $i$ in $\cc H$ and the edges of 
type $u^{h-i}$ in $\cc F$ for each $i\in[h]$. We color the edges of type $u^i$ in $\cc F$ with the same color as the corresponding edge in $\cc H$ for $0\leq i\leq h-1$. By Corollary \ref{maindetachcor}, if we can color the remaining edges of $\cc F$ (edges of type $u^h$) so that the following condition is satisfied, then we are done. 
\begin{equation}  \label{maindetachcoreqr}
\forall i\in [k] \quad \quad  \deg_{\cc F(i)}(v) = \left \{ \begin{array}{ll}
r & \mbox { if } v\neq u,  \\
rm & \mbox { if } v=u. \end{array} \right.\\
\end{equation}

Let $\mult_i(u^j,.)$ be the number of edges of  type $u^j$ in $\cc F(i)$, for $i\in [k], j\in [h]$. Note that $\mult_i(u^h,.)=\mult_{\cc F(i)}(u^h)$ for $i\in [k]$. 
 We color the edges of type $u^h$ so that for $i\in [k]$,
$$
\mult_i(u^h,.)=\frac{rn}{h}-r(n-m)+\sum_{j=1}^{h-1} j\mult_i(u^{h-j-1},.).
$$ 
Since $h| rn$,  $\mult_i(u^h,.)$ is an integer for $i\in [k]$. The following shows that $\mult_i(u^h,.)\geq 0$ for $i\in [k]$.
\begin{eqnarray*}
\frac{rn}{h}&\mathop\geq \limits^{\scriptsize (\ref{neccpiece})}& |E(\cc H(i))|=\sum_{j=0}^{h-1} \mult_i(u^j,.)\\
&=&\sum_{j=1}^{h} j\mult_i(u^{h-j},.)-\sum_{j=1}^{h-1} j\mult_i(u^{h-j-1},.)\\
&\mathop= \limits^{\scriptsize (\ref{neccpiece0})}&r(n-m)-\sum_{j=1}^{h-1} j\mult_i(u^{h-j-1},.).
\end{eqnarray*}

Now we show  that all edges of the type $u^h$ will be colored, or equivalently that, $\sum_{i=1}^k \mult_i(u^h,.)=\binom{m}{h}$. 
\begin{eqnarray*}
\sum_{i=1}^k \mult_i(u^h,.)&=&\sum_{i=1}^k \big(\frac{rn}{h}-r(n-m)+\sum_{j=1}^{h-1} j\mult_i(u^{h-j-1},.)\big)\\
&=& \frac{rkn}{h}-rk(n-m)+\sum_{j=1}^{h-1} j\sum_{i=1}^k \mult_i(u^{h-j-1},.)\\
&=&\binom{n}{h}-(n-m)\binom{n-1}{h-1}+\sum_{j=2}^{h} (j-1)\binom{m}{h-j}\binom{n-m}{j}\\
&\mathop= \limits^{ (\ref{easydc1}), (\ref{easydc2})} &\sum_{j=0}^h \binom{m}{j}\binom{n-m}{h-j}-\sum_{j=1}^{h-1} j\binom{n-m}{j}\binom{m}{h-j}\\
&-&(n-m)\binom{n-m-1}{h-1}+\sum_{j=2}^{h} (j-1)\binom{m}{h-j}\binom{n-m}{j}\\
&=&\binom{m}{h}-(n-m)\binom{n-m-1}{h-1}+h\binom{n-m}{h}\\
&=&\binom{m}{h}.
\end{eqnarray*}
To complete the proof, we show that $\deg_{\cc F(i)}(u)=rm$ for $i\in[k]$. We have
\begin{eqnarray*}
\deg_{\cc F(i)}(u)&=&\sum_{j=1}^h j\mult_i(u^j,.)=h\mult_i(u^h,.)+\sum_{j=1}^h (h-j)\mult_i(u^{h-j},.)\\
&=&h\mult_i(u^h,.)+h\sum_{j=1}^h \mult_i(u^{h-j},.)-\sum_{j=1}^h j\mult_i(u^{h-j},.)\\
&=&h\sum_{j=0}^h \mult_i(u^{h-j},.)-\sum_{j=1}^h j\mult_i(u^{h-j},.)
\\
&=&rn-r(n-m)=rm.
\end{eqnarray*}
\end{proof}
For a hypergraph $\cc G$ and $V\subseteq V(\cc G)$, let $\cc G\backslash V$ be the hypergraph whose vertex set is $V(\cc G)$ and whose edge set is $\{e\in E(\cc G)|e\nsubseteq V\}$.

Let $V\subseteq V(K_n^h)$ with $|V|=m\leq n$, and let $\cc H:=K_n^h\backslash V$. An edge $e\in E(\cc H)$ is of {\it type} $i$, if $|e\cap V|=i$ (for $0\leq i\leq h-1$). Let $P$ be a  partial $r$-factorization of $\cc H$. Then  a  partial $r$-factorization $Q$ of $\cc H$ is said to be {\it $P$-friendly} if 
\begin{enumerate}[(a)]
\item the color of each edge of type 0 is the same in $P$ and $Q$, and 
\item the number of edges of type $i$ and color $j$ is the same in $P$ and $Q$ for each $i\in[h-1]$ and each color $j$.
\end{enumerate}
We are interested in finding the conditions under which a partial $r$-factorization of $\cc H$ can be extended to an $r$-factorization of $K_n^h$.
\begin{lemma} \label{necceasyh}
For $V\subseteq V(K_n^h)$ with $|V|=m$, if a partial $r$-factorization of $\cc H:=K_n^h\backslash V$ can be extended to an $r$-factorization of $K_n^h$, then 
\begin{enumerate}
\item [(N1)] $h|rn$,
\item [(N2)] $r| \binom{n-1}{h-1}$, 
\item [(N3)] $d_{\cc H(i)}(v)= r$ for each $v\in V(\cc H)\backslash V$, and $i\in [k]$,
\item [(N4)] $|E(\cc H(i))|\leq rn/h$ for $i\in [k]$,
\end{enumerate}
where $k:=\binom{n-1}{h-1}/r$.
\end{lemma}
It remains an open question whether these conditions are sufficient. Here we prove a weaker result.
\begin{corollary} \label{corpieces}
Let $V\subseteq V(K_n^h)$ with $|V|=m$, and let $P$ be a partial $r$-factorization  of $\cc H:=K_n^h\backslash V$, and assume that (N1)--(N4) are satisfied. Then there exists a $P$-friendly partial $r$-factorization of $\cc H$ that can be extended to an $r$-factorization of $K_n^h$. 
\end{corollary}
\begin{proof}
By eliminating all the vertices in $V$, and shrinking the edges containing vertices in $V$, we obtain $K_n^h-V$. The rest of the proof follows from Theorem \ref{piecesthm}.  
\end{proof} 
 
\section{$h=4$}\label{h4emb}
\begin{theorem}
For $n\geq 4.847323m$, any partial $r$-factorization of $K_m^4$ can be extended to an $r$-factorization of $K_n^4$ if and only if $4 | rn$ and $r | \binom{n-1}{3}$.
\end{theorem}
\begin{proof}
For the necessary conditions, see the previous section. To prove the sufficiency,  we need to show that the edges of $\cc F:=\widetilde {K_m^4}$ can be colored with  $k:=\binom{n-1}{3}/r$ colors so that 
  (\ref{maindetachcoreqr}) is satisfied. 

First we color the edges in $\cc F$ of the form $W\cup \{u\}$ where $W\subseteq V:=V(K_m^4)$ and $|W|=3$. We color these edges greedily so that $\deg_i(x)\leq r$ for each $x\in V$ and $i\in [k]$. We claim that this coloring can be done in such a way that all edges of this type are colored. Suppose by contrary that there is an edge in $\cc F$ of the form $\{x,y,z,u\}$ with $x,y,z\in V$ that can not be colored. This implies that for each $i\in [k]$ either $\deg_i(x)=r$ or $\deg_i(y)=r$ or $\deg_i(z)=r$. Thus for each $i\in [k]$,    $\deg_i(x)+\deg_i(y)+\deg_i(z)\geq r$. On the one hand, $\sum_{i=1}^k \big(\deg_i(x)+\deg_i(y)+\deg_i(z) \big)\geq rk=\binom{n-1}{3}$, and on the other hand, $ \sum_{i=1}^k \big(\deg_i(x)+\deg_i(y)+\deg_i(z) \big)\leq 3[\binom{m-1}{3}+(n-m)\binom{m-1}{2}-1]$. Thus, we have 
$$
3[\binom{m-1}{3}+(n-m)\binom{m-1}{2}-1]\geq \binom{n-1}{3}.
$$
which is equivalent to $f(n,m):=n^3-6 n^2-9 m^2 n+27 m n-7 n+6 m^3-9 m^2-15 m+30\leq 0$. Now, we show that since $n>4m$ and  $m\geq 5$, we have $f(n,m)>0$, which is a contradiction, and therefore, all edges in $\cc F$ of the form $W\cup \{u\}$ where $W\subseteq V$ and $|W|=3$ can be colored using the greedy approach described above. 

First, note that for $m\geq 5$, both $7m^2+3m-7$ and $2m^2-3m-5$ are positive. Therefore,  
\begin{eqnarray*}
f(n,m)&=&n\Big(n(n-6)-9m^2+27m-7\Big)+3m(2m^2-3m-5)+30\\
&>&n\Big(4m(4m-6)-9m^2+27m-7\Big)+3m(2m^2-3m-5)+30\\
&=&n(7m^2+3m-7)+3m(2m^2-3m-5)+30>0.
\end{eqnarray*}

Now we greedily color all the edges of the form $W\cup \{u^2\}$ where $W\subseteq V$ and $|W|=2$, so that  $\deg_i(x)\leq r$ for each $x\in V$ and $i\in [k]$. We show that this is possible. If by contrary, some edge $\{x,y,u^2\}$ with $x,y\in V$ remains uncolored, then  for each $i\in [k]$ either $\deg_i(x)=r$ or $\deg_i(y)=r$, and so  $\deg_i(x)+\deg_i(y)\geq r$. We have
\begin{eqnarray*}
\binom{n-1}{3}=rk&\leq& \sum_{i=1}^k \big(\deg_i(x)+\deg_i(y)\big) \\
&\leq& 2[\binom{m-1}{3}+(n-m)\binom{m-1}{2}+(m-1)\binom{n-m}{2}-1],\nonumber
\end{eqnarray*}
which is equivalent to $n^3-6 m n^2+6 m^2 n+12 m n-7 n-2 m^3-6 m^2-4 m+18\leq 0$. Using Mathematica (Wolfram Alpha) it can be shown that this inequality does not have any real  solution under the constraints $m\geq 5, n \geq 4.847323m$. Therefore, all edges of the form $W\cup \{u^2\}$ where $W\subseteq V$ and $|W|=2$, can be colored.

Since for each $x\in V$, 
\begin{eqnarray*}
\sum_{i=1}^k \big(r-\deg_i(x)\big)&=&rk-[\binom{m-1}{3}+(n-m)\binom{m-1}{2}+(m-1)\binom{n-m}{2}]\\
&=&\binom{n-1}{3}-\binom{m-1}{3}-(n-m)\binom{m-1}{2}-(m-1)\binom{n-m}{2}\\
&\mathop= \limits^{ (\ref{easydc1})} &\binom{n-m}{3},
\end{eqnarray*}
we can color all the edges of the form $\{w,u^3\}$ where $w\in V$  so that for each $x\in V$, there are $r-\deg_i(x)$ edges of this type colored $i$ incident with $x$ for each $i\in [k]$. Note that after this coloring, 
\begin{equation} \label{cond1sat}
\deg_i(x)=r \mbox{ for each } x\in V.
\end{equation}
For $i\in [k]$, let $a_i, b_i, c_i, d_i$ be the number of edges  colored $i$ of the form $W, W\cup\{u\},W\cup\{u^2\},W\cup\{u^3\}$ where  $W\subseteq V$, respectively. 
We color the edges of the form $\{u^4\}$ so that there are exactly $$e_i:=rn/4-rm+3a_i+2b_i+c_i$$ edges of this type colored $i$ for $i\in [k]$. Since $4| rn$, and $n> 4m$, $e_i$ is a positive integer for $i\in [k]$. We claim that all edges of the form 	$\{u^4\}$ will be colored, or equivalently, $\sum_{i=1}^k e_i=\binom{n-m}{4}$. 

\begin{eqnarray*}
\sum_{i=1}^k e_i&=&\sum_{i=1}^k (\frac{rn}{4}-rm+3a_i+2b_i+c_i)=\frac{rkn}{4}-rkm+3\sum_{i=1}^k a_i+2\sum_{i=1}^k b_i+\sum_{i=1}^k c_i\\
&=&\frac{n}{4}\binom{n-1}{3}-m\binom{n-1}{3}+3\binom{m}{4}+2(n-m)\binom{m}{3}+\binom{n-m}{2}\binom{m}{2}\\
&=&\binom{n}{4}-m\binom{n-1}{3}+3\binom{m}{4}+2(n-m)\binom{m}{3}+\binom{n-m}{2}\binom{m}{2}\\
&\mathop= \limits^{ (\ref{easydc1}), (\ref{easydc2})} &\binom{n-m}{4}.
\end{eqnarray*}

To complete the proof, we show that $\deg_i(u)=r(n-m)$ for $i\in[k]$. First note that for  $i\in [k]$, $rm=\sum_{x\in V} \deg_i(x)=4a_i+3b_i+2c_i+d_i$. Therefore, 
\begin{eqnarray*}
\deg_i(u)=b_i+2c_i+3d_i+4e_i&=&4(a_i+b_i+c_i+d_i+e_i)-(4a_i+3b_i+2c_i+d_i)\\
&=&rn-rm=r(n-m).
\end{eqnarray*}
Combining this with (\ref{cond1sat}) implies  that  (\ref{maindetachcoreqr}) is satisfied, and the proof is complete.
\end{proof}

\section{$h=5$}\label{h5emb}
\begin{theorem}
For $n\geq6.285214m$, any partial $r$-factorization of $K_m^5$ can be extended to an $r$-factorization of $K_n^5$ if and only if $5 | rn$ and $r | \binom{n-1}{4}$.
\end{theorem}
\begin{proof}
The necessity is obvious. To prove the sufficiency,  we need to show that the edges of $\cc F:=\widetilde {K_m^5}$ can be colored with  $k:=\binom{n-1}{4}/r$ colors so that 
  (\ref{maindetachcoreqr}) is satisfied.

First we color the edges of the form $W\cup \{u\}$ where $W\subseteq V$ and $|W|=4$. We color these edges greedily so that $\deg_i(x)\leq r$ for each $x\in V$ and $i\in [k]$. We claim that this coloring can be done in such a way that all edges of this type are colored. Suppose by contrary that there is an edge of the form $\{x,y,z,w,u\}$ with $x,y,z,w\in V$ that can not be colored. This implies that for each $i\in [k]$ either $\deg_i(x)=r$ or $\deg_i(y)=r$ or $\deg_i(z)=r$ or $\deg_i(w)=r$. Thus for each $i\in [k]$,    $\deg_i(x)+\deg_i(y)+\deg_i(z)+\deg_i(w)\geq r$. On the one hand, $\sum_{i=1}^k \big(\deg_i(x)+\deg_i(y)+\deg_i(z)+\deg_i(w) \big)\geq rk=\binom{n-1}{4}$, and on the other hand, $ \sum_{i=1}^k \big(\deg_i(x)+\deg_i(y)+\deg_i(z)+\deg_i(w) \big)\leq 4[\binom{m-1}{4}+(n-m)\binom{m-1}{3}-1]$. Thus, we have 
$$
4[\binom{m-1}{4}+(n-m)\binom{m-1}{3}-1]\geq \binom{n-1}{4}.
$$
which is equivalent to $g_1(n,m):=n^4-10 n^3+35 n^2-16 m^3 n+96 m^2 n-176 m n+46 n+12 m^4-56 m^3+36 m^2+104 m+24\leq 0$. 

Since $n>6m$ and  $m\geq 6$, we have 
\begin{eqnarray*}
g_1(n,m)&:=&n\Big(n^2(n-10)-16m^3+96m^2+(35n-176m)+46\Big)\\
&+& 4m\Big(m^2(3m-14)+9m+26\Big)+24\\
&>& 9m^2(3m-10)-16m^3+96m^2\\
&=& 11m^3+6m^2>0,
\end{eqnarray*}
which is a contradiction, and therefore, all edges in $\cc F$ of the form $W\cup \{u\}$ where $W\subseteq V$ and $|W|=4$ can be colored.

Now we greedily color all the edges of the form $W\cup \{u^2\}$ where $W\subseteq V$ and $|W|=3$, so that  $\deg_i(x)\leq r$ for each $x\in V$ and $i\in [k]$. We show that this is possible. If by contrary, some edge $\{x,y,z,u^2\}$ with $x,y,z\in V$ remains uncolored, then  for each $i\in [k]$ either $\deg_i(x)=r$ or $\deg_i(y)=r$ or $\deg_i(z)=r$, and so  $\deg_i(x)+\deg_i(y)+\deg_i(z)\geq r$. We have
\begin{eqnarray*}
\binom{n-1}{4}=rk&\leq& \sum_{i=1}^k \big(\deg_i(x)+\deg_i(y)+\deg_i(z)\big) \\
&\leq& 3[\binom{m-1}{4}+(n-m)\binom{m-1}{3}+\binom{m-1}{2}\binom{n-m}{2}-1].\nonumber
\end{eqnarray*}
which is equivalent to $g_2(n,m):=n^4-10 n^3-18 m^2 n^2+54 m n^2-n^2+24 m^3 n-18 m^2 n-114 m n+58 n-9 m^4-6 m^3+45 m^2+42 m+24\leq 0$.  We show that since $n>6m$ and  $m\geq 6$, we have $g_2(n,m)>0$, which is a contradiction, and therefore, all edges in $\cc F$ of the form $W\cup \{u^2\}$ where $W\subseteq V$ and $|W|=3$ can be colored.

First, note that for $m\geq 6$ we have $12m^3-9m^2-57m+29>0$. Therefore,  
\begin{eqnarray*}
g_2(n,m)&=&n^2\Big(n(n-10)-18m^2+54m-1\Big)\\
&+&2n\Big(12m^3-9m^2-57m+29\Big)\\
&-& (9 m^4+6 m^3-45 m^2-42 m-24)\\
&>& 36m^2\Big(6m(6m-10)-18m^2+54m-1\Big)\\
&-& (9 m^4+6 m^3-45 m^2-42 m-24)\\
&=&639m^4-150m^3+18m^2+42m+24>0.
\end{eqnarray*}

Now we greedily color all the edges of the form $W\cup \{u^3\}$ where $W\subseteq V$ and $|W|=2$, so that  $\deg_i(x)\leq r$ for each $x\in V$ and $i\in [k]$. We show that this is possible. If by contrary, some edge $\{x,y,u^2\}$ with $x,y\in V$ remains uncolored, then  for each $i\in [k]$ either $\deg_i(x)=r$ or $\deg_i(y)=r$, and so  $\deg_i(x)+\deg_i(y)\geq r$. We have
\begin{eqnarray*}
\binom{n-1}{4} &\leq& \sum_{i=1}^k \big(\deg_i(x)+\deg_i(y)\big) \\
&\leq& 2[\binom{m-1}{4}+(n-m)\binom{m-1}{3}+\binom{m-1}{2}\binom{n-m}{2}+(m-1)\binom{n-m}{3}-1].\nonumber
\end{eqnarray*}
Using Mathematica it can be shown that this inequality does not have any real  solution under the constraints $m\geq 6, n \geq 6.285214m$. Therefore, all edges of the form $W\cup \{u^3\}$ where $W\subseteq V$ and $|W|=2$, can be colored.

Since for each $x\in V$, 
\begin{eqnarray*}
\sum_{i=1}^k \big(r-\deg_i(x)\big)&=&\binom{n-1}{4}-\binom{m-1}{4}-(n-m)\binom{m-1}{3}\\
&-&\binom{m-1}{2}\binom{n-m}{2}-(m-1)\binom{n-m}{3}\\
&=&\binom{n-m}{4},
\end{eqnarray*}
we can color all the edges of the form $\{w,u^4\}$ where $w\in  V$  so that for each $x\in V$, there are $r-\deg_i(x)$ edges of this type colored $i$ incident with $x$ for each $i\in [k]$.

For $i\in [k]$, let $a_i, b_i, c_i, d_i, e_i$ be the number of edges  colored $i$ of the form $W, W\cup\{u\},W\cup\{u^2\},W\cup\{u^3\},W\cup\{u^4\}$ where  $W\subseteq V$, respectively. 
We color the edges of the form $\{u^5\}$ so that there exactly $$f_i:=rn/5-rm+4a_i+3b_i+2c_i+d_i$$ edges of this type colored $i$ for $i\in [k]$. Since $5| rn$, and $n\geq 6.4m>5m$, $e_i$ is a positive  integer for $i\in [k]$. We claim that all edges of the form 	$\{u^5\}$ will be colored, or equivalently, $\sum_{i=1}^k f_i=\binom{n-m}{5}$. 

\begin{eqnarray*}
\sum_{i=1}^k f_i&=&\sum_{i=1}^k (\frac{rn}{5}-rm+4a_i+3b_i+2c_i+d_i)\\
&=&\frac{rkn}{5}-rkm+4\sum_{i=1}^k a_i+3\sum_{i=1}^k b_i+2\sum_{i=1}^k c_i+\sum_{i=1}^k d_i\\
&=&\binom{n}{5}-m\binom{n-1}{4}+4\binom{m}{5}+3(n-m)\binom{m}{4}+2\binom{n-m}{2}\binom{m}{3}+\binom{n-m}{3}\binom{m}{2}\\
&=&\binom{n-m}{5}.
\end{eqnarray*}

To complete the proof, we show that $\deg_i(u)=r(n-m)$ for $i\in[k]$. First note that for  $i\in [k]$, $rm=\sum_{x\in V} \deg_i(x)=5a_i+4b_i+3c_i+2d_i+e_i$. Therefore, 
\begin{eqnarray*}
\deg_i(u)&=&b_i+2c_i+3d_i+4e_i+5f_i\\
&=&5(a_i+b_i+c_i+d_i+e_i+f_i)-(5a_i+4b_i+3c_i+2d_i+e_i)\\
&=&rn-rm=r(n-m).
\end{eqnarray*}
\end{proof}

\section{Concluding Remarks and Open Problems}
\begin{enumerate}
\item At this point, it is not clear to use how to extend the results of Sections \ref{h4emb} and \ref{h5emb} without dealing with heavy computation. We believe  that for $n\geq 2hm$, any partial $r$-factorization of $K_m^h$ can be extended to an $r$-factorization of $K_n^h$ if and only if the obvious necessary divisibility conditions are satisfied. 
\item To embed a partial $r$-factorization of $K_n\backslash K_m^h$ into an $r$-factorization of $K_n^h$, we believe that  the conditions (N1)--(N4) of Lemma \ref{necceasyh} are  sufficient, but we do not know how to go beyond Corollary \ref{corpieces}. 

\item  A partial $r$-factorization $S\subseteq K_n^h$ is {\it critical} if it can be extended to exactly one $r$-factorization  of  $K_n^h$, but removal of any element of $S$ destroys the uniqueness of the extension, and $|S|$ is the {\it size} of the critical partial $r$-factorization. It is desirable to find  good bounds  for the smallest and largest sizes of critical partial $r$-factorizations. 

\item Another interesting problem is  finding conditions under which a partial $r$-factorization of  $S\subseteq \binom{[n]}{h}$ can  be extended to a {\bf cyclic} $r$-factorization of  $\binom{[n]}{h}$. 
\end{enumerate}

\section*{Acknowledgement}
 The  author's research is  supported by Summer Faculty Fellowship at ISU, and NSA Grant H98230-16-1-0304. The author wishes to thank Lana Kuhle,  Dan Roberts and the anonymous referees for their constructive feedback on the first draft of this paper.

\bibliographystyle{plain}

\end{document}